\documentclass{article}

\usepackage[a4paper,margin=1in]{geometry}
\usepackage{amsmath,amssymb,amsthm,mathtools}
\usepackage{microtype}
\usepackage{aliascnt}
\usepackage[numbers,sort]{natbib}
\usepackage[hidelinks]{hyperref}
\usepackage[nameinlink,capitalise,noabbrev]{cleveref}
\usepackage{enumitem}
\usepackage{tikz}
\usepackage[affil-it]{authblk}
\usetikzlibrary{arrows.meta,positioning,calc,decorations.pathreplacing}
\definecolor{class0}{RGB}{210,228,247}
\definecolor{class1}{RGB}{214,238,214}
\definecolor{class2}{RGB}{250,229,205}
\hypersetup{
  pdftitle={Near-Optimal Covering Sequences},
  pdfauthor={Hoang Ta}
}
\usepackage{booktabs,tabularx}

\newaliascnt{lemma}{theorem}
\newtheorem{lemma}[lemma]{Lemma}
\aliascntresetthe{lemma}
\newaliascnt{proposition}{theorem}
\newtheorem{proposition}[proposition]{Proposition}
\aliascntresetthe{proposition}
\newaliascnt{corollary}{theorem}

\aliascntresetthe{corollary}
\theoremstyle{definition}
\newaliascnt{definition}{theorem}
\newtheorem{definition}[definition]{Definition}
\aliascntresetthe{definition}
\theoremstyle{remark}
\newaliascnt{remark}{theorem}

\aliascntresetthe{remark}
\newtheorem*{theorem*}{\textbf{Main Theorem}}

\newaliascnt{example}{theorem}
\newtheorem{example}[example]{Example}
\aliascntresetthe{example}

\newcommand{\Zq}{\mathbb{Z}_q}
\newcommand{\Sigmaq}{\Sigma_q}
\newcommand{\cF}{\mathcal{F}}
\newcommand{\cP}{\mathcal{P}}
\newcommand{\cC}{\mathcal{C}}
\newcommand{\cL}{\mathcal{L}}
\newcommand{\dist}{d_{\mathrm H}}

\DeclareMathOperator{\len}{len}
\DeclareMathOperator{\lcm}{lcm}
\DeclareMathOperator{\Fix}{Fix}
\newcommand{\bigO}{\mathcal{O}}
\newcommand{\cD}{\mathcal{D}}

\begin{document}

\title{{\LARGE Near-Optimal Covering Sequences}}

\author[1]{Hoang Ta}
\author[2]{Van Khu Vu}

\affil[1]{\small Hanoi University of Science and Technology, Vietnam}
\affil[2]{\small VinUniversity, Vietnam}
\date{\today}

\maketitle

\begin{abstract}
An $(n,R)$-covering sequence over a finite alphabet $\Sigmaq \coloneqq \{0,1,\dots, q-1\}$ is a cyclic sequence whose consecutive length-$n$ windows form a covering code of radius $R$. Equivalently, every word in $\Sigmaq^n$ is within Hamming distance $R$ of at least one window. We give a deterministic and explicit construction of such sequences whose length, for every fixed alphabet size $q$, every fixed radius $R$, and every sufficiently large $n$, attains the sphere-covering lower bound up to a constant factor depending only on $q$ and $R$. Thus, in the fixed-radius regime, the construction removes the logarithmic factor in the general probabilistic upper bounds of [Chung and Cooper, \emph{Random Structures \& Algorithms}, 2004] and [Vu, \emph{Advances in Applied Mathematics}, 2005]. It also complements the earlier explicit constructions of [Chee, Etzion, Ta, and Vu, \emph{Designs, Codes and Cryptography}, 2025], which include constant factor bounds for the special binary radius-one families \(n=2^a-1\) and \(n=2^a\), where \(a\ge1\).

\end{abstract}

\section{Introduction}\label{sec:intro}

Covering codes are classical objects in coding theory and may be viewed as a dual counterpart of error-correcting codes. In a covering code, the aim is to choose a small set of codewords such that every word in the ambient Hamming space is close to at least one chosen codeword. Such codes have been extensively studied and have connections to data compression, combinatorial search, and finite geometry; see ~\cite{CohenEtAl1985,CohenLobsteinSloane1986,GrahamSloane1985,CohenEtAl1997}. A basic benchmark for this problem is the sphere-covering bound, obtained by comparing the size of the ambient space with the size of a Hamming ball.

Covering sequences impose an additional cyclic constraint on this covering problem. Instead of choosing arbitrary covering centres, one asks that the
centres arise as consecutive windows of a single cyclic sequence. Thus the problem is not only to cover the Hamming space, but also to arrange the
covering centres in a cyclic-window structure. This extra constraint is the main distinction between ordinary covering codes and covering sequences.

The case of zero covering radius recovers classical de Bruijn sequences. A de Bruijn sequence of order \(n\) over an alphabet of size \(q\) is a cyclic
sequence in which each word of length \(n\) occurs exactly once as a cyclic window. Such sequences exist for all \(q\ge2\) and \(n\ge1\) ~\cite{deBruijn1946,vanAardenneDeBruijn1951,Fredricksen1982}; see also~\cite{Etzion2024Book} for further background on de Bruijn sequences and the
de Bruijn graph. Beyond their combinatorial significance, de Bruijn sequences and related shift-register sequences also appear in information-processing applications, including quantum communication and cryptographic sequence generation ~\cite{Golomb1982,Mitchell1997,ZhangEtAl2021,CheeEtAl2022RLL,MandalGong2012,YangEtAl2017}. Covering sequences may therefore be regarded as approximate de Bruijn sequences, where exact occurrence of every length-\(n\) word is replaced by coverage within a prescribed Hamming radius.

The notion of covering sequences was introduced by Chung and Cooper~\cite{ChungCooper2004}, who obtained a probabilistic upper bound for
prime-power alphabets. This bound was subsequently extended to arbitrary finite alphabets by Vu~\cite{Vu2005}. Namely, for every fixed alphabet size \(q\)
and fixed covering radius \(R\), these results give covering sequences of
length
\[
  \bigO_{q,R}\!\left(\frac{q^n\log n}{n^R}\right),
\]
which is within a logarithmic factor of the sphere-covering lower bound. The framework of covering sequence codes, together with explicit constructions, interleaving methods, and two-dimensional analogues, was further developed in~\cite{CheeEtAl2025}. In particular, that work gives constant factor
constructions for the special binary radius-one families \(n=2^a-1\) and \(n=2^a\). The aim of the present paper is to remove the logarithmic factor in the fixed-\(q\), fixed-\(R\) regime by constructing, for every sufficiently large \(n\), covering sequences whose lengths match the sphere-covering lower bound up to a constant factor.

\subsection{Preliminaries}
\label{sec:preliminaries}
Throughout the paper, $q\ge 2$ and $R\ge 1$ denote fixed integers. We identify an alphabet of size $q$ with the additive ring $\Sigmaq=\Zq=\mathbb{Z}/q\mathbb{Z}$. A \emph{cyclic sequence} of length $\ell$ is a tuple $A=(a_t)_{t\in\mathbb{Z}_{\ell}}$ with $a_t\in\Sigmaq$, whose coordinates are indexed by the cyclic group $\mathbb{Z}_{\ell}$; accordingly, all indices are read modulo $\ell$. We write $\len(A)=\ell$ for its length. For an integer phase $j$ and a positive integer $u$, the \emph{cyclic window} of length $u$ starting at phase $j$ is defined by
\begin{equation}\label{eq:window}
  A[j,u]:=(a_j,a_{j+1},\ldots,a_{j+u-1})\in\Sigmaq^{u},
\end{equation}
where every subscript is reduced modulo $\ell$. For a single coordinate we abbreviate $A[j]:=a_j$. We stress that the window length $u$ may exceed
$\ell$, in which case the sequence is simply read periodically. This convention is used repeatedly in what follows and is indispensable when a sequence is shorter than the span it is required to cover. For two words $x,y$ of equal length, $\dist(x,y)$ denotes their \emph{Hamming distance},
that is, the number of coordinates in which they differ. The asymptotic notation refers to the regime in which the main length
parameter tends to infinity. Constants implicit in \(\bigO_{q,R}(\cdot)\) and \(\Theta_{q,R}(\cdot)\), and the rates of convergence in \(o_{q,R}(1)\), may depend only on the displayed subscripted parameters. The analogous convention applies to subscripts \(q\) and \(R\).

\begin{definition}\label{def:cs}
A cyclic sequence $S$ over $\Sigmaq$ is called a \emph{$q$-ary $(n,R)$-covering sequence}, or equivalently a \emph{de Bruijn covering sequence of span $n$ and radius $R$}, if for every target word $y\in\Sigmaq^n$ there exists a phase $j$ such that $ \dist\bigl(y,S[j,n]\bigr)\le R.$ We write $\cL_q(n,R)$ for the smallest length of such a cyclic sequence.
\end{definition}

\begin{definition}\label{def:csc}
An indexed finite family $\cF$ of cyclic sequences over $\Sigmaq$ is called a \emph{$q$-ary $(n,R)$-covering sequence code}, abbreviated $(n,R)$-CSC, if
every target word $y\in\Sigmaq^n$ lies within Hamming distance $R$ of some $n$-window of at least one member of $\cF$. The cost of such a family is
measured by two parameters,
\[
  Q(\cF):=\sum_{A\in\cF}\len(A)
  \qquad\text{and}\qquad
  D(\cF):=|\cF|,
\]
called its \emph{total cyclic length} and its \emph{number of components}, respectively.
\end{definition}

\begin{example}
    Let \(q=2\), \(n=3\), and \(R=1\). Consider the indexed family
\[
  \cF=(A_0,A_1),
  \qquad
  A_0=(0,0,1),\quad A_1=(0,1,1),
\]
where both sequences are read cyclically. The length-\(3\) windows of \(A_0\)
are
\[
  001,\quad 010,\quad 100,
\]
and the length-\(3\) windows of \(A_1\) are
\[
  011,\quad 110,\quad 101.
\]
Thus the windows of the two components are precisely the binary words of length \(3\) having Hamming weight \(1\) or \(2\). Every binary word of length \(3\) is within Hamming distance at most \(1\) from one of these windows: the words \(001,010,100,011,101,110\) are covered exactly, while \(000\) is at distance \(1\) from any word of weight \(1\), and \(111\) is at distance \(1\) from any word of weight \(2\). Hence \(\cF\) is a binary \((3,1)\)-CSC. Its total cyclic length and number of components are
\[
  Q(\cF)=\len(A_0)+\len(A_1)=3+3=6,
  \qquad
  D(\cF)=2.
\]
\end{example}

For $x\in\Sigmaq^n$, let $B_R(x):=\{\,y\in\Sigmaq^n:\dist(x,y)\le R\,\}$ denote the Hamming ball of radius $R$ centred at $x$. Its cardinality is
independent of $x$ and equals
\[
  V_q(n,R):=|B_R(x)|=\sum_{j=0}^{R}\binom{n}{j}(q-1)^j .
\]
Let $K_q(n,R)$ denote the minimum size of a $q$-ary covering code of length $n$ and radius $R$. Since the Hamming balls of radius $R$ centred at the
codewords are required to cover the entire space $\Sigmaq^n$, the standard sphere-covering argument yields
\[
  K_q(n,R)\ge \frac{q^n}{V_q(n,R)} .
\]
The same lower bound carries over to covering sequences, the length-$n$ windows of an $(n,R)$-covering sequence constitute an $(n,R)$-covering code, and
consequently
\[
  \cL_q(n,R)\ge K_q(n,R)\ge \frac{q^n}{V_q(n,R)} .
\]
For fixed $q$ and $R$,
\begin{equation}\label{eq:ball-asymptotic}
  V_q(n,R)=\frac{(q-1)^R}{R!}\,n^R+\bigO_{q,R}\!\left(n^{R-1}\right),
\end{equation}
so that this lower bound is of order $q^n/n^R$.

For ordinary covering codes, the sphere-covering lower bound is known to be tight up to a constant factor for every fixed \(q\) and \(R\): there exists a
constant \(c_{q,R}\), independent of \(n\), such that
\[
  K_q(n,R)\le c_{q,R}\,\frac{q^n}{V_q(n,R)} \, ,
\]
for all sufficiently large \(n\); see~\cite{CohenEtAl1985,CohenEtAl1997,KrivelevichSudakovVu2003}.

In terms of the Hamming-ball volume defined above, the probabilistic bounds of Chung and Cooper~\cite{ChungCooper2004} and Vu~\cite{Vu2005} can be written as
\[
  \cL_q(n,R)
  \le
  \bigO_{q,R}\!\left(
    \frac{q^n}{V_q(n,R)}\log n
  \right).
\]
The main result below removes the factor \(\log n\) for every fixed \(q\) and \(R\).


We also use the following two-dimensional version. Let $m,n,M,N$ be positive integers. A doubly periodic $q$-ary array of period $M\times N$ is a tuple
\[
  A=(a_{u,v})_{(u,v)\in\mathbb Z_M\times\mathbb Z_N},
  \qquad a_{u,v}\in\Sigmaq,
\]
and its area is $MN$. For a phase $(i,j)\in\mathbb Z_M\times\mathbb Z_N$, the $m\times n$ window of $A$ at $(i,j)$ is
\[
A[(i,j);m,n]
:=
(a_{i+r,j+s})_{\substack{0\le r<m \\ 0\le s<n}}
\in\Sigmaq^{m\times n},
\]
where the two indices are reduced modulo $M$ and $N$, respectively. When measuring Hamming distance, we identify an $m\times n$ array with a word of length $mn$. The array $A$ is called a \emph{$q$-ary $(m\times n,R)$-covering two-dimensional sequence}, or simply an \emph{$(m\times n,R)$-covering 2D sequence}, if every $Y\in\Sigmaq^{m\times n}$ lies within Hamming distance $R$ of some window $A[(i,j);m,n]$. We then set
\[
\cL_q^{(2)}(m,n,R)
:= \min\bigl\{\,MN:\ \text{some $(m\times n,R)$-covering 2D sequence has period
  $M\times N$}\,\bigr\}.
\]
Since an $M\times N$ array has $MN$ window phases, the sphere-covering argument
yields
\[
\cL_q^{(2)}(m,n,R)
\ge
\frac{q^{mn}}{V_q(mn,R)}.
\]
Two-dimensional covering sequences and their connection to the one-dimensional case were studied in~\cite{CheeEtAl2025}, where folding constructions were used to convert one-dimensional covering sequences into two-dimensional covering sequences with only a constant factor loss in area.

\subsection{Main results and techniques.}

In this paper, we give a deterministic and explicit construction of \(q\)-ary covering sequences whose length, for every sufficiently
large span \(n\), lies within a constant factor of the sphere-covering bound, where the constant depends only on \(q\) and \(R\). The construction proceeds in two stages. We first build a radius-one covering sequence code: the feedback coefficients of a linear-feedback shift register are read off a de Bruijn sequence, which forces every syndrome to occur as a column of the parity-check matrix and thereby guarantees covering radius one. Since the leading coefficient is a unit, the register is reversible, so its state space splits into disjoint cycles, and a short-cycle estimate shows that almost all states lie on long cycles; hence the resulting family has few components. We then raise the radius from one to $R$ by interleaving $R$ such families along every diagonal phase orbit, which removes any coprimality hypothesis on the component lengths, and a standard linearisation step finally merges the family into a single cyclic covering sequence. These ideas yield the following bound.

\begin{theorem*}
For fixed integers $q\ge 2$ and $R\ge 1$, as $n\to\infty$, the minimum length of a $q$-ary $(n,R)$-covering sequence of span $n$ and radius $R$ satisfies
\[
  \cL_q(n,R)
  \le
  \left(2q^R R^{R+1}+o_{q,R}(1)\right)\frac{q^n}{n^R}
  =
  \left(
    \frac{2q^R(q-1)^R R^{R+1}}{R!}
    +o_{q,R}(1)
  \right)
  \frac{q^n}{V_q(n,R)} .
\]
\end{theorem*}

Together with the sphere-covering lower bound $ \cL_q(n,R)\ge \frac{q^n}{V_q(n,R)},$ the main result determines the minimum length of a $q$-ary $(n,R)$-covering sequence up to a constant factor depending only on $q$ and $R$. In particular, for fixed $q$ and $R$,
\[
  \cL_q(n,R)
  =\Theta_{q,R}\!\left(\frac{q^n}{V_q(n,R)}\right)
  =\Theta_{q,R}\!\left(\frac{q^n}{n^R}\right),
\]
where the second equality uses $V_q(n,R)=\Theta_{q,R}(n^R)$. The upper bound is obtained through an explicit deterministic construction and holds for every sufficiently large span $n$. The result also admits a two-dimensional consequence, applying the folding construction of~\cite{CheeEtAl2025} to the covering sequences constructed here yields, for fixed $q$ and $R$,
\[
  \cL_q^{(2)}(m,n,R)
  \le
  \bigO_{q,R}\!\left(\frac{q^{mn}}{V_q(mn,R)}\right)
  =
  \bigO_{q,R}\!\left(\frac{q^{mn}}{(mn)^R}\right).
\]
This removes the logarithmic factor in the previous general probabilistic upper bound for two-dimensional covering sequences, namely
\[
  \cL_q^{(2)}(m,n,R)
  \le
  \bigO_{q,R}\!\left(
    \frac{q^{mn}}{V_q(mn,R)}\bigl(\log m+\log n\bigr)
  \right),
\]
obtained in~\cite{CheeEtAl2025}. Table~\ref{tab:known-upper-bounds} summarizes the relevant upper bounds.

\begin{table}[t]
\centering
\footnotesize
\setlength{\tabcolsep}{5pt}
\renewcommand{\arraystretch}{1.2}
\begin{tabularx}{\textwidth}{
  @{}
  |c
  |>{\raggedright\arraybackslash}p{2.7cm}
  |c
  |c
  |c
  |>{\raggedright\arraybackslash}X|
  @{}
}
\hline
Dimension
& Reference
& Alphabet
& Radius
& Admissible sizes
& \multicolumn{1}{c|}{Upper bound}
\\
\hline
1D
& Chung--Cooper \cite{ChungCooper2004};
  Vu \cite{Vu2005}
& \(q\ge 2\)
& fixed \(R\)
& all \(n\)
&
\(\displaystyle
  \bigO_{q,R}\!\left(
    \frac{q^n}{V_q(n,R)}\log n
  \right)\)
\\
\hline
1D
& Chee--Etzion--Ta--Vu \cite{CheeEtAl2025}
& \(q=2\)
& \(R=1\)
& \(n=2^a-1 \text{ or }  2^a\)
&
\(\displaystyle
  \bigO\!\left(\frac{2^n}{n}\right)\)
\\
\hline
\textbf{1D}
& \textbf{This paper}
& \(\boldsymbol{q\ge 2}\)
& \textbf{fixed} \(\boldsymbol{R}\)
& \textbf{all} \(\boldsymbol{n}\)
&
\(\displaystyle
  \boldsymbol{\bigO_{q,R}\!\left(
    \frac{q^n}{V_q(n,R)}
  \right)}\)
\\
\hline\hline
2D
& Chee--Etzion--Ta--Vu \cite{CheeEtAl2025}
& \(q\ge 2\)
& fixed \(R\)
& all \(m,n\)
&
\(\displaystyle
  \bigO_{q,R}\!\left(
    \frac{q^{mn}}{V_q(mn,R)}
    (\log m+\log n)
  \right)\)
\\
\hline
\textbf{2D}
& \textbf{This paper}
& \(\boldsymbol{q\ge 2}\)
& \textbf{fixed} \(\boldsymbol{R}\)
& \textbf{all} \(\boldsymbol{m,n}\)
&
\(\displaystyle
  \boldsymbol{\bigO_{q,R}\!\left(
    \frac{q^{mn}}{V_q(mn,R)}
  \right)}\)
\\
\hline
\end{tabularx}
\caption{Known upper bounds for covering sequences. A bound is called constant factor optimal here if it matches the sphere-covering lower bound \(q^n/V_q(n,R)\) in 1D, or \(q^{mn}/V_q(mn,R)\) in 2D, up to a multiplicative constant depending only on \(q\) and \(R\). Unless specific lengths are listed, bounds hold for all sufficiently large \(n\) (1D) and \(mn\) (2D).}
\label{tab:known-upper-bounds}
\end{table}

\section{The radius-one construction}\label{sec:radius1}

Throughout this section, let \(m\ge q+1\) be an integer. We first construct a radius-one covering sequence code over \(\Zq\) with few components. The
parameters of the construction will be fixed below. The construction combines a coefficient block copied from a de Bruijn sequence, a radius-one additive code over \(\Zq\), and a reversible feedback shift register. A final estimate on short state cycles will control the concatenation cost in \cref{sec:radiusR}.

We use the following classical existence lemma for de Bruijn sequences; see~\cite{deBruijn1946,vanAardenneDeBruijn1951,Fredricksen1982,Etzion2024Book} for background and constructions.

\begin{lemma}\label{lem:debruijn}
For every $q\ge 2$ and $r\ge 1$, there exists a cyclic sequence $B=(b_t)_{t\in\mathbb{Z}_{q^r}}$ of length $q^r$ such that the map
\[
  t\longmapsto (b_t,b_{t+1},\ldots,b_{t+r-1})
\]
is a bijection from $\mathbb{Z}_{q^r}$ onto $\Zq^r$. Equivalently, every $r$-tuple over $\Zq$ occurs exactly once as a cyclic window of $B$.
\end{lemma}

We use this sequence to fix the parameters of the construction. Set
\begin{equation}\label{eq:def-r-k}
  r=r_q(m):=\max\bigl\{\,j\ge 1:\ q^j+2j-1\le m\,\bigr\},
  \qquad k:=m-r.
\end{equation}

The choice of \(r\) means that the coefficient block of length \(q^r+r-1\) used below fits inside the register of length \(k=m-r\). In particular,
\(r=\bigO_q(\log m)\), and hence \(k=m-r\) differs from \(m\) only by a logarithmic term. By the maximality of \(r\) in \cref{eq:def-r-k}, one has
\begin{equation}\label{eq:r-inequalities}
  q^r+2r-1\le m
  \qquad\text{and}\qquad
  q^{r+1}+2r+1>m.
\end{equation}
The first inequality ensures that the coefficient sequence used below fits inside the register. Indeed,
\begin{equation}\label{eq:coefficient-fits}
  q^r+r-2\le k-1,
\end{equation}
because of $ k-1=m-r-1\ge (q^r+2r-1)-r-1=q^r+r-2.$

Fix a de Bruijn sequence \(B=(b_t)_{t\in\mathbb Z_{q^r}}\) whose cyclic windows of length \(r\) contain every element of \(\Zq^r\) exactly once. Since every \(r\)-tuple occurs, the symbol \(1\) occurs somewhere in \(B\). After a cyclic shift, we may therefore assume that \(b_0=1\). We extend the indexing periodically by setting \(b_{t+q^r}=b_t\), and define coefficients \(a_0,\ldots,a_{k-1}\in\Zq\) by
\begin{equation}\label{eq:coefficients}
  a_j:=
  \begin{cases}
    b_j, & 0\le j\le q^r+r-2,\\[2pt]
    0,   & q^r+r-1\le j\le k-1.
  \end{cases}
\end{equation}
The bound \cref{eq:coefficient-fits} guarantees that this definition is valid: the copied part from \(B\) ends no later than position \(k-1\), and
any remaining positions are filled with zeros. In particular, \(a_0=b_0=1\). We now use these coefficients as parity-check coefficients. Define the
additive code as follows.
\begin{equation}\label{eq:def-Cm}
  \cC_m:=
  \left\{
    x=(x_0,\ldots,x_{m-1})\in\Zq^m:\
    x_{t+k}+\sum_{j=0}^{k-1}a_jx_{t+j}=0
    \ \text{for }0\le t<r
  \right\}.
\end{equation}
Thus \(\cC_m\) is the \(\Zq\)-submodule of \(\Zq^m\) consisting of all words satisfying these \(r\) consecutive parity checks, with all arithmetic
performed in the ring \(\Zq\). We next determine its cardinality and covering radius. The key point for the covering bound is that the coefficient block copied from the de Bruijn sequence contains, as consecutive length-\(r\) windows, all vectors of \(\mathbb Z_q^r\). This makes every syndrome appear as a column of the parity-check matrix.


\begin{proposition}\label{prop:radius-one-code}
The code $\cC_m$ has cardinality $q^k$ and covering radius at most 1 in
$\Zq^m$.
\end{proposition}

\begin{proof}
\emph{Cardinality.} Fix arbitrary symbols $x_0,\ldots,x_{k-1}\in\Zq$ for the first $k$ coordinates. The defining constraint \cref{eq:def-Cm} at $t=0$
expresses $x_k$ uniquely in terms of these symbols; the constraint at $t=1$ then determines $x_{k+1}$, and proceeding inductively through $t=r-1$ fixes
the remaining coordinate $x_{m-1}$. Consequently, each of the $q^k$ admissible choices for the first $k$ coordinates extends to exactly one codeword, whence $|\cC_m|=q^k$.

\emph{Covering radius.} Let $H$ be the $r\times m$ parity-check matrix over $\Zq$ associated with \cref{eq:def-Cm}, so that $\cC_m=\ker H$. Indexing rows
by $t=0,\ldots,r-1$ and columns by $i=0,\ldots,m-1$, the entries of $H$ are given by
\begin{equation}\label{eq:Hentries}
  H_{t,i}=
  \begin{cases}
    a_{i-t}, & 0\le i-t<k,\\
    1,       & i=t+k,\\
    0,       & \text{otherwise},
  \end{cases}
\end{equation}
where the two nonzero cases are mutually exclusive. Denote by $h_i\in\Zq^r$ the $i$-th column of $H$.

Consider first a column index $i$ satisfying $r-1\le i\le k-1$. As $t$ ranges over $\{0,\ldots,r-1\}$, the offset $i-t$ ranges over $\{i-r+1,\ldots,i\}\subseteq\{0,\ldots,k-1\}$; hence every entry of the column falls into the first case of \cref{eq:Hentries}, while the symbol $1$ never
occurs since $i\le k-1<t+k$. The column therefore takes the explicit form
\begin{equation}\label{eq:internal-column}
  h_i=(a_i,a_{i-1},\ldots,a_{i-r+1})^{\mathsf T}.
\end{equation}
Specialising to $i=r-1+t$ with $0\le t<q^r$, \cref{eq:coefficient-fits} guarantees that these indices lie in $[r-1,k-1]$, and every coefficient
appearing in \cref{eq:internal-column} then has index at most $q^r+r-2$ and thus coincides with the corresponding symbol of $B$. We conclude that
\begin{equation}\label{eq:columns-debruijn}
  h_{r-1+t}=(b_{t+r-1},b_{t+r-2},\ldots,b_t)^{\mathsf T},
\end{equation}
which is precisely the coordinate reversal of the length-\(r\) cyclic window \((b_t,\ldots,b_{t+r-1})\) of the de Bruijn sequence \(B\). By \cref{lem:debruijn}, this window traverses every
element of $\Zq^r$ exactly once as $t$ runs through $\{0,\ldots,q^r-1\}$; since coordinate reversal is a bijection of $\Zq^r$, it follows that
\begin{equation}\label{eq:all-syndromes-columns}
  \bigl\{\,h_{r-1+t}:0\le t<q^r\,\bigr\}=\Zq^r.
\end{equation}
In other words, \emph{every syndrome in $\Zq^r$ is realised as a column of $H$}.

Now let $y\in\Zq^m$ be arbitrary and set $\sigma:=Hy$. If $\sigma=0$, then $y\in\cC_m$ already. Otherwise $\sigma\ne 0$, and \cref{eq:all-syndromes-columns} furnishes an index $i$ with $h_i=\sigma$. Writing $e_i$ for the $i$-th standard basis vector of $\Zq^m$, put $y':=y-e_i$. Since $1\ne0$ in $\Zq$, the vectors $y$ and $y'$ differ in the single coordinate $i$, and
\[
  Hy'=Hy-He_i=\sigma-h_i=0,
\]
so that $y'\in\cC_m$ with $\dist(y,y')=1$. Every word of $\Zq^m$ thus lies within Hamming distance $1$ of $\cC_m$, and so $\cC_m$ has covering radius at
most $1$, completing the proof.
\end{proof}

The proof uses only addition in \(\Zq\) and the fact that \(1\) is a unit; it never divides by an arbitrary nonzero symbol. Hence
\cref{prop:radius-one-code} holds for every \(q\ge2\), composite or not. We now realise \(\cC_m\) dynamically. Consider the state-transition map
\begin{equation}\label{eq:transition}
  T:\Zq^k\longrightarrow\Zq^k,\qquad
  T(u_0,\ldots,u_{k-1}):=
  \left(u_1,\ldots,u_{k-1},\ -\sum_{j=0}^{k-1}a_j u_j\right).
\end{equation}
This is the standard linear-feedback shift register with coefficients \(a_0,\ldots,a_{k-1}\); see~\cite{Golomb1982,Fredricksen1982}. It discards the leading symbol, shifts the rest, and appends the feedback value \(-\sum_j a_j u_j\); see \cref{fig:lfsr}.

\begin{lemma}\label{lem:T-permutation}
The map $T$ is a permutation of the state space $\Zq^k$.
\end{lemma}

\begin{proof}
We establish bijectivity by exhibiting the inverse map explicitly. Let \(v=(v_0,\ldots,v_{k-1})=T(u)\), where \(u=(u_0,\ldots,u_{k-1})\). By the
definition of \(T\), the first \(k-1\) coordinates of \(v\) reproduce the trailing coordinates of \(u\):
\[
  v_0=u_1,\quad v_1=u_2,\quad \ldots,\quad v_{k-2}=u_{k-1}.
\]
Consequently, the symbols \(u_1,\ldots,u_{k-1}\) are already recovered from \(v\), and it remains only to determine \(u_0\). The final coordinate of \(v\)
satisfies
\[
  v_{k-1}
  =-\sum_{j=0}^{k-1}a_j u_j
  =-u_0-\sum_{j=1}^{k-1}a_j u_j,
\]
where the second equality uses \(a_0=1\). Substituting \(u_j=v_{j-1}\) for \(1\le j\le k-1\) and solving for \(u_0\) yields \(u_0=-v_{k-1}-\sum_{j=1}^{k-1}a_jv_{j-1}\). All arithmetic is performed in \(\Zq\). The vector \(u\) is therefore uniquely
determined by \(v\), and the inverse of \(T\) is given explicitly by
\[
  T^{-1}(v_0,\ldots,v_{k-1})
  =
  \left(
    -v_{k-1}-\sum_{j=1}^{k-1}a_jv_{j-1},
    v_0,\ldots,v_{k-2}
  \right).
\]
Since \(T\) admits a two-sided inverse, it is a bijection, and hence a permutation of \(\Zq^k\).
\end{proof}

For a state $u\in\Zq^k$ let $x_t(u)$ denote the first coordinate of $T^t(u)$ for $t\ge 0$. Iterating the shift in \cref{eq:transition} gives, for every
$t\ge 0$, the state-output identity
\begin{equation}\label{eq:state-output}
  T^t(u)=\bigl(x_t(u),x_{t+1}(u),\ldots,x_{t+k-1}(u)\bigr)
\end{equation}
together with the linear recurrence
\begin{equation}\label{eq:global-recurrence}
  x_{t+k}(u)+\sum_{j=0}^{k-1}a_j x_{t+j}(u)=0\qquad(t\ge 0).
\end{equation}
Because $T$ is a permutation of a finite set, the state space decomposes into a disjoint union of directed cycles. For each state cycle $\gamma$ of length
$\ell$, choose and fix once and for all a base state $u_\gamma\in\gamma$, and record the cyclic output sequence
\begin{equation}\label{eq:output-cycle}
  A_\gamma:=
  \bigl(x_0(u_\gamma),x_1(u_\gamma),\ldots,x_{\ell-1}(u_\gamma)\bigr)
\end{equation}
of length $\ell$. This fixed choice is needed because, under our definition, a cyclic sequence is an indexed tuple rather than an equivalence class modulo
rotation. If \(u=T^d(u_\gamma)\) for an integer \(d\in\{0,\ldots,\ell-1\}\), then
\begin{equation}\label{eq:phase-shift-output}
  x_t(u)=x_{t+d}(u_\gamma)\qquad(t\ge0),
\end{equation}
where the right-hand sequence is read periodically. Thus changing the base state rotates the tuple, while the collection of its cyclic windows remains
unchanged.

\begin{figure}[ht]
\centering
\begin{tikzpicture}[font=\small, >=Stealth,
  cell/.style={draw, minimum width=9mm, minimum height=7mm},
  dot/.style={circle, fill, inner sep=1.2pt}]
\node[cell] (c0) at (0,0) {$x_t$};
\node[cell] (c1) at (1.05,0) {$x_{t+1}$};
\node (cd) at (2.05,0) {$\cdots$};
\node[cell] (c3) at (3.15,0) {$x_{t+k-1}$};
\node[cell, fill=black!8] (c4) at (4.55,0) {$x_{t+k}$};
\draw[->] (c0.west) -- ++(-0.6,0) node[left]{out};
\node[draw, rounded corners, minimum width=44mm, minimum height=7mm]
  (sum) at (1.75,-1.55) {$-\sum_{j=0}^{k-1} a_j\,x_{t+j}$};
\draw[->] (c0.south) -- (c0.south |- sum.north) node[midway,left=-1pt]{\scriptsize$a_0$};
\draw[->] (c1.south) -- (c1.south |- sum.north) node[midway,left=-1pt]{\scriptsize$a_1$};
\draw[->] (c3.south) -- (c3.south |- sum.north) node[midway,right=-1pt]{\scriptsize$a_{k-1}$};
\draw[->] (sum.east) -| (c4.south);
\node[font=\footnotesize] at (1.9,0.95) {(a) shift register $T$};
\begin{scope}[shift={(7.8,-0.35)}]
  \foreach \i in {1,...,7} { \node[dot] (a\i) at ({360/7*\i}:0.85) {}; }
  \foreach \i in {1,...,7} { \pgfmathtruncatemacro{\j}{mod(\i,7)+1}
    \draw[->] (a\i) to[bend left=9] (a\j); }
  \foreach \i in {1,...,5} { \node[dot] (b\i) at ($(2.25,-0.05)+({360/5*\i}:0.7)$) {}; }
  \foreach \i in {1,...,5} { \pgfmathtruncatemacro{\j}{mod(\i,5)+1}
    \draw[->] (b\i) to[bend left=11] (b\j); }
  \node[dot] (s1) at (1.15,-1.55) {};
  \node[dot] (s2) at (1.6,-1.55) {};
  \draw[->] (s1) to[bend left=45] (s2);
  \draw[->] (s2) to[bend left=45] (s1);
  \node[font=\footnotesize] at (1.15,1.2) {long cycles};
  \node[font=\footnotesize] at (1.4,-2.0) {short cycle};
  \node[font=\footnotesize] at (1.15,-2.7) {(b) state space of $T$};
\end{scope}
\end{tikzpicture}
\caption{The radius-one construction. (a)~The feedback shift register \(T\) appends \(x_{t+k}=-\sum_{j=0}^{k-1}a_jx_{t+j}\) as the state shifts left. The length-\(m\) windows of its output streams are exactly the codewords of \(\cC_m\) (\cref{prop:radius-one-CSC}). (b)~Since \(a_0=1\), the transition map \(T\) is invertible (\cref{lem:T-permutation}), and its state space decomposes into cycles. The short-cycle estimate in \cref{lem:short-cycles} implies that almost all states lie on long cycles, so the resulting output family has few components.}
\label{fig:lfsr}
\end{figure}
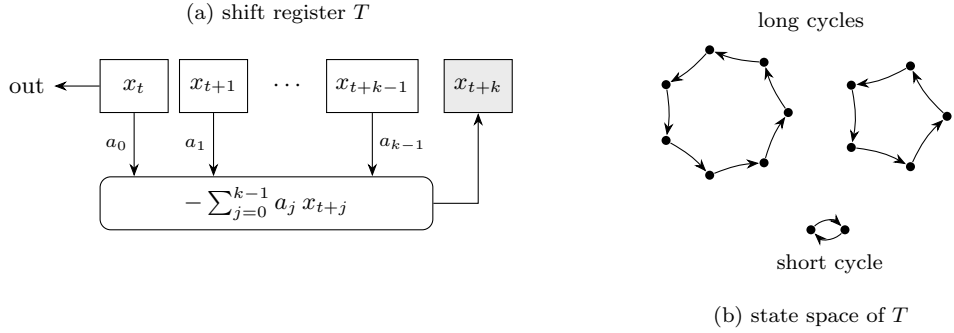

Collecting these output sequences over all state cycles yields the radius-one family, and the next proposition identifies its windows with the code
$\cC_m$, thereby transferring the covering property of
\cref{prop:radius-one-code}.

\begin{proposition}\label{prop:radius-one-CSC}
Let \(\cF_m=(A_\gamma)_\gamma\) be the indexed family of output sequences, where \(\gamma\) ranges over the state cycles of \(T\). Then \(\cF_m\) is an
\((m,1)\)-CSC, and its total cyclic length is
\begin{equation}\label{eq:Fm-total-length}
  Q(\cF_m)=q^k.
\end{equation}
\end{proposition}

\begin{proof}
\emph{Every window is a codeword.} Fix a state cycle \(\gamma\) of length \(\ell\) and an integer \(d\in\{0,\ldots,\ell-1\}\), and set \(u:=T^d(u_\gamma)\). By \cref{eq:phase-shift-output},
\[
  A_\gamma[d,m]
  =\bigl(x_0(u),x_1(u),\ldots,x_{m-1}(u)\bigr).
\]
The recurrence \cref{eq:global-recurrence} then ensures that this window satisfies the $r$ checks of \cref{eq:def-Cm}. The conclusion persists even
when $\ell<m$: since $T^\ell(u)=u$, the output sequence emanating from $u$ is $\ell$-periodic, precisely as demanded by the cyclic-window convention.
Consequently, every $m$-window of every member of $\cF_m$ belongs to $\cC_m$.

\emph{Every codeword is a window.} Conversely, let $c=(c_0,\ldots,c_{m-1})\in\cC_m$ and take the initial state $u=(c_0,\ldots,c_{k-1})$. Starting from this state, the recurrence \cref{eq:global-recurrence} generates the symbols $x_k(u),\ldots,x_{m-1}(u)$. As $c$ satisfies the same $r$ check equations, induction on each newly generated coordinate yields
\[
  \bigl(x_0(u),\ldots,x_{m-1}(u)\bigr)=c.
\]
The state $u$ lies on a unique cycle $\gamma$, and relative to the fixed base state $u_\gamma$ there is a unique integer \(d\in\{0,\ldots,\ell-1\}\) for which \(u=T^d(u_\gamma)\). Invoking \cref{eq:phase-shift-output} once more gives
\(c=A_\gamma[d,m]\). The collection of all $m$-windows occurring in $\cF_m$ therefore coincides exactly with $\cC_m$, which has covering radius at most one by \cref{prop:radius-one-code}; this establishes the CSC property. Finally, the state cycles partition the $q^k$ states and $\len(A_\gamma)=|\gamma|$, whence $Q(\cF_m)=\sum_\gamma|\gamma|=q^k$.
\end{proof}

The next estimate is the main quantitative input for the construction. It shows that almost all states lie on \emph{long} cycles, and hence that the
family \(\cF_m\) has few components.

\begin{lemma}\label{lem:short-cycles}
For every positive integer $h$, the total length of those members of $\cF_m$ whose length is strictly less than $h$ is at most
\begin{equation}\label{eq:short-total}
  \sum_{\ell=1}^{h-1}q^\ell<\frac{q^h}{q-1}.
\end{equation}
Consequently the number of components satisfies
\begin{equation}\label{eq:number-components-general}
  D(\cF_m)\le\frac{q^k}{h}+\frac{q^h}{q-1}.
\end{equation}
\end{lemma}

\begin{proof}
For $\ell\ge 1$, consider the fixed-point set $\Fix(T^\ell)=\{u\in\Zq^k:T^\ell(u)=u\}$. Whenever $u\in\Fix(T^\ell)$, one has $x_{t+\ell}(u)=x_t(u)$ for every $t\ge 0$. By \cref{eq:state-output}, the initial block of $\ell$ output symbols already pins down the state: if $\ell<k$, the remaining state coordinates arise by periodic repetition, while if $\ell\ge k$ the state is itself contained among the first $k$ of these symbols. The map $u\mapsto(x_0(u),\ldots,x_{\ell-1}(u))$ is therefore injective on $\Fix(T^\ell)$, and consequently
\begin{equation}\label{eq:fix-bound}
  |\Fix(T^\ell)|\le q^\ell.
\end{equation}
Since every state lying on a cycle of length exactly $\ell$ belongs to $\Fix(T^\ell)$, at most $q^\ell$ states lie on length-$\ell$ cycles. Summing
this bound over $1\le\ell<h$ yields \cref{eq:short-total}, the total length of the short components.

Each of the remaining cycles has length at least $h$; as the cycles jointly account for all $q^k$ states, there can be at most $q^k/h$ such long cycles.
The number of short cycles is in turn bounded by their total length, hence by \cref{eq:short-total}. Adding the two estimates gives \cref{eq:number-components-general}.
\end{proof}

We now record the asymptotic consequences of the radius-one construction. These estimates will be used in the radius-\(R\) construction, and they also
show that the radius-one case already has the correct order of magnitude.

\begin{lemma}\label{lem:radius-one-asymptotics}
Write \(Q_m:=Q(\cF_m)=q^{\,m-r_q(m)}\). Then, as \(m\to\infty\), we have
\begin{align}
  \frac{q^m}{m}
  &\le Q_m
  \le \left(q+o_q(1)\right)\frac{q^m}{m},
  \label{eq:Qm-bounds}\\[2pt]
  D(\cF_m)
  &\le \left(1+o_q(1)\right)\frac{Q_m}{m}.
  \label{eq:Dm-asymptotic}
\end{align}
Consequently, the radius-one covering sequence length satisfies
\begin{equation}\label{eq:L-radius-one-bound}
  \cL_q(m,1)
  \le
  \left(2q+o_q(1)\right)\frac{q^m}{m}.
\end{equation}
\end{lemma}
\begin{proof}
We begin by estimating the total cyclic length. The inequality \(q^r\le m\) from \cref{eq:r-inequalities} immediately yields
\[
  Q_m=\frac{q^m}{q^r}\ge \frac{q^m}{m}.
\]
For the matching upper bound, the second inequality in \cref{eq:r-inequalities} provides \(q^r>(m-2r-1)/q\), so that, for all sufficiently large \(m\),
\begin{equation}\label{eq:Q-upper-explicit}
  Q_m<\frac{q^{m+1}}{m-2r-1}.
\end{equation}
Because \(r=\bigO_q(\log m)\), we have \(m-2r-1=(1-o_q(1))m\), and therefore
\[
  Q_m\le \left(q+o_q(1)\right)\frac{q^m}{m}.
\]
This establishes \cref{eq:Qm-bounds}.

We turn next to the number of components, applying \cref{lem:short-cycles} with \(h=\lfloor m-\sqrt m\rfloor\) to obtain
\[
  D(\cF_m)\le \frac{Q_m}{h}+\frac{q^h}{q-1}.
\]
Since \(h/m\to 1\), the first term satisfies
\[
  \frac{Q_m}{h}
  =
  \left(1+o_q(1)\right)\frac{Q_m}{m},
\]
while the second term is negligible relative to \(Q_m/m\); indeed,
\[
  \frac{q^h}{Q_m/m}
  =
  m q^{\,h-(m-r)}
  \le
  m q^{\,r-\sqrt m}
  \longrightarrow 0,
\]
again because \(r=\bigO_q(\log m)\). Combining these two estimates gives
\[
  D(\cF_m)
  \le
  \left(1+o_q(1)\right)\frac{Q_m}{m},
\]
which proves \cref{eq:Dm-asymptotic}.

Finally, since \(\cF_m\) is an \((m,1)\)-CSC, the standard linearisation and concatenation procedure for covering sequence codes \cite[Section~3]{CheeEtAl2025}, recalled later in \cref{lem:concatenation}, produces a single \((m,1)\)-covering sequence of length at most
\[
  Q_m+(m-1)D(\cF_m).
\]
Invoking \cref{eq:Dm-asymptotic}, this quantity is bounded by
\[
  Q_m+(m-1)\left(1+o_q(1)\right)\frac{Q_m}{m}
  =
  \left(2+o_q(1)\right)Q_m.
\]
Together with the upper bound in \cref{eq:Qm-bounds}, we conclude that
\[
  \cL_q(m,1)
  \le
  \left(2q+o_q(1)\right)\frac{q^m}{m},
\]
as claimed.
\end{proof}

\section{The general radius-\texorpdfstring{$R$}{R} construction}
\label{sec:radiusR}

We now extend the construction from radius one to an arbitrary fixed radius \(R\). The underlying idea is to take \(R\) radius-one covering families and
interleave them according to the residue classes modulo \(R\). Since each residue class contributes at most one error, every interleaved window is then
covered within radius \(R\). This interleaving step again produces a cyclic family, which is subsequently converted into a single covering sequence by
linearisation and concatenation.

Interleaving constructions for covering sequences were employed in~\cite[Section~5]{CheeEtAl2025} under coprimality assumptions on the component
lengths; it is instructive to recall the basic two-sequence situation. Let \(A\) and \(B\) be \((n_1,R_1)\)- and \((n_2,R_2)\)-covering sequences of
lengths \(\ell_1\) and \(\ell_2\), respectively, where
\[
  n_1=n_2
  \qquad\text{or}\qquad
  n_1=n_2+1.
\]
When \(\gcd(\ell_1,\ell_2)=1\), the standard alternating interleaving of \(A\) and \(B\) constitutes an \((n_1+n_2,R_1+R_2)\)-covering sequence of length
\(2\ell_1\ell_2\), the coprimality condition ensuring that every pair of starting phases in \(\mathbb Z_{\ell_1}\times\mathbb Z_{\ell_2}\) is realised as
the interleaved sequence is read cyclically.

When \(\gcd(\ell_1,\ell_2)>1\), by contrast, a single interleaving no longer visits every phase pair. Instead, the diagonal action on
\(\mathbb Z_{\ell_1}\times\mathbb Z_{\ell_2}\) partitions the phase space into \(d=\gcd(\ell_1,\ell_2)\) orbits, each of which yields one interleaved cyclic sequence of length
\[
  2\lcm(\ell_1,\ell_2),
\]
so that assembling all \(d\) orbits produces a covering sequence code of total cyclic length
\[
  d\cdot 2\lcm(\ell_1,\ell_2)=2\ell_1\ell_2.
\]
The loss of coprimality is thus accommodated by replacing the single interleaved sequence with a family indexed by the diagonal phase orbits,
without any increase in the total cyclic length. The construction developed below applies precisely this diagonal-orbit viewpoint to \(R\) component
sequences.

Write $n=Rm+s,\, 0\le s<R.$ The positions of a length-\(n\) word then partition into \(R\) residue classes modulo \(R\): the first \(s\) classes each contain \(m+1\) positions, while the remaining \(R-s\) classes each contain \(m\) positions. Accordingly, set
\begin{equation}\label{eq:mi-definition}
  m_i:=
  \begin{cases}
    m+1, & 0\le i<s,\\
    m,   & s\le i<R,
  \end{cases}
  \qquad\text{so that}\qquad \sum_{i=0}^{R-1}m_i=n.
\end{equation}

Since \(R\) is fixed and \(m=\lfloor n/R\rfloor\to\infty\), we have \(m_i\ge q+1\) for every \(0\le i<R\) and all sufficiently large \(n\). Thus the radius-one construction of \cref{sec:radius1} applies with span \(m_i\) for each residue class \(i\). We denote the resulting family by
\[
  \cF_i:=\cF_{m_i},\qquad Q_i:=Q(\cF_i),
  \qquad P:=\prod_{i=0}^{R-1}Q_i.
\]

We now describe how to interleave one component drawn from each of these \(R\) families. Fix a tuple
\[
  \boldsymbol{A}=(A_0,\ldots,A_{R-1})
  \in \cF_0\times\cdots\times\cF_{R-1},
\]
and write \(\ell_i:=\len(A_i)\). A choice of starting phases for this tuple is an element of the product
\[
  G_{\boldsymbol{A}}:=
  \mathbb{Z}_{\ell_0}\times\cdots\times\mathbb{Z}_{\ell_{R-1}}.
\]
Advancing an interleaved sequence by one full round of \(R\) symbols advances each component phase by one, so the relevant phase evolution is the diagonal
translation
\begin{equation}\label{eq:diagonal-action}
  \tau(\alpha_0,\ldots,\alpha_{R-1})
  :=
  (\alpha_0+1,\ldots,\alpha_{R-1}+1).
\end{equation}

The order of this translation is
\begin{equation}\label{eq:L-lcm}
  L:=\lcm(\ell_0,\ldots,\ell_{R-1}).
\end{equation}
All additions in the \(i\)-th coordinate are taken modulo \(\ell_i\). For
\(\boldsymbol\alpha\in G_{\boldsymbol A}\), define its diagonal orbit by
\[
  \operatorname{Orb}(\boldsymbol\alpha)
  :=
  \{\,\tau^t(\boldsymbol\alpha):0\le t<L\,\}.
\]
Thus every diagonal orbit has length \(L\), and the number of such orbits is $ \frac{\prod_{i=0}^{R-1}\ell_i}{L}.$

For each diagonal orbit \(\cD\subseteq G_{\boldsymbol A}\), choose its lexicographically least representative
\(\boldsymbol{\alpha}=(\alpha_0,\ldots,\alpha_{R-1})\), and define a cyclic sequence \(W_{\boldsymbol A,\cD}\) of length \(RL\) by
\begin{equation}\label{eq:interleaved-word}
  W_{\boldsymbol{A},\cD}[\,Rt+i\,]
  :=
  A_i[\,\alpha_i+t\,],
  \qquad 0\le t<L,\quad 0\le i<R.
\end{equation}
Here the index on the left is read modulo \(RL\), while the phase on the right is read modulo \(\ell_i\). The definition is consistent because
\(\ell_i\mid L\) for every \(i\), so that replacing \(t\) by \(t+L\) leaves the symbol \(A_i[\alpha_i+t]\) unchanged. Equivalently, the subsequence of
\(W_{\boldsymbol{A},\cD}\) occupying the positions congruent to \(i\pmod R\) is the cyclic sequence \(A_i\) read from phase \(\alpha_i\).

Finally, let \(\cP_n\) denote the indexed family of all sequences \(W_{\boldsymbol{A},\cD}\), where \(\boldsymbol{A}\) ranges over
\(\cF_0\times\cdots\times\cF_{R-1}\) and \(\cD\) ranges over all diagonal orbits in \(G_{\boldsymbol{A}}\). The construction is illustrated in
\cref{fig:interleave} for \(R=3\).

It is precisely the use of all diagonal orbits that removes any coprimality requirement on the component lengths. Were one to fix only a single phase
vector, the resulting interleaving would visit nothing beyond the diagonal orbit of that vector, and in general such a sequence would fail to realise
every combination of component phases unless the lengths satisfied suitable coprimality assumptions. By admitting one interleaved sequence for each
diagonal orbit, every phase vector is represented somewhere in the family. As shown in \cref{prop:interleaving-size}, the number of orbits and the common
orbit length cancel in the total-length calculation, so that this enlargement of the family incurs no increase in the total cyclic length.

\begin{figure}[ht]
\centering
\begin{tikzpicture}[font=\footnotesize, >=Stealth,
  w/.style={draw, minimum width=13mm, minimum height=8mm, font=\scriptsize}]
\node[w, fill=class0] (w0) at (0,0)     {$A_0[\alpha_0]$};
\node[w, fill=class1] (w1) at (1.36,0)  {$A_1[\alpha_1]$};
\node[w, fill=class2] (w2) at (2.72,0)  {$A_2[\alpha_2]$};
\node[w, fill=class0] (w3) at (4.08,0)  {$A_0[\alpha_0{+}1]$};
\node[w, fill=class1] (w4) at (5.44,0)  {$A_1[\alpha_1{+}1]$};
\node[w, fill=class2] (w5) at (6.80,0)  {$A_2[\alpha_2{+}1]$};
\node[w, fill=class0] (w6) at (8.16,0)  {$A_0[\alpha_0{+}2]$};
\node[left=2mm of w0] {$W:$};
\draw[decorate,decoration={brace,amplitude=4pt,mirror}]
  (w0.south west) -- (w2.south east) node[midway,below=5pt]{round $0$};
\draw[decorate,decoration={brace,amplitude=4pt,mirror}]
  (w3.south west) -- (w5.south east) node[midway,below=5pt]{round $1$};
\draw[decorate,decoration={brace,amplitude=4pt,mirror}]
  (w6.south west) -- (w6.south east) node[midway,below=5pt]{partial};
\node[anchor=west] at (-0.3,1.25)
  {residue class $i\pmod R\;\to\;$ a window of $A_i$\, (length $m{+}1$ if $i<s$, else $m$)};
\end{tikzpicture}
\caption{Diagonal interleaving for $R=3$ and $n=7$ (so $m=2$, $s=1$). The interleaved sequence reads the components round by round,
$W[Rt+i]=A_i[\alpha_i+t]$. The positions in residue class $i$ modulo $R$ trace out a window of $A_i$, one symbol longer for the first $s$ classes, so, each class being covered to within one error, the length-$n$ window of $W$ is covered to within $R$ errors (\cref{prop:interleaving-cover}).}
\label{fig:interleave}
\end{figure}
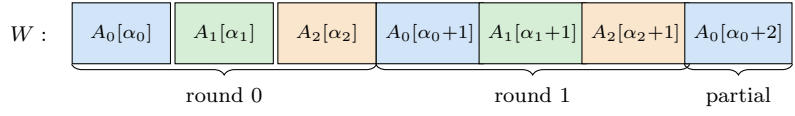

We first establish the covering property of \(\cP_n\), and then derive the corresponding bounds on its total length and number of components.

\begin{proposition}\label{prop:interleaving-cover}
The family $\cP_n$ is an $(n,R)$-CSC.
\end{proposition}

\begin{proof}
Let \(y=(y_0,\ldots,y_{n-1})\in\Zq^n\) be an arbitrary target word. For each \(0\le i<R\), extract the subword supported on the coordinates congruent to
\(i\) modulo \(R\):
\begin{equation}\label{eq:residue-subword}
  y^{(i)}:=(y_i,y_{i+R},\ldots,y_{i+(m_i-1)R})\in\Zq^{m_i}.
\end{equation}
The number of indices in \(\{0,\ldots,n-1\}\) congruent to \(i\) modulo \(R\) equals
\[
  \left\lfloor\frac{n-1-i}{R}\right\rfloor+1,
\]
which is \(m+1\) for \(0\le i<s\) and \(m\) for \(s\le i<R\); consequently
\(y^{(i)}\) has length \(m_i\), in accordance with \cref{eq:mi-definition}.

Since \(\cF_i\) is an \((m_i,1)\)-CSC, there exist a component \(A_i\in\cF_i\) and, writing \(\ell_i:=\len(A_i)\), a phase \(\beta_i\in\mathbb Z_{\ell_i}\) for which
\begin{equation}\label{eq:component-cover}
  \dist\bigl(y^{(i)},A_i[\beta_i,m_i]\bigr)\le 1.
\end{equation}
Set
\[
  \boldsymbol A:=(A_0,\ldots,A_{R-1}),
  \qquad
  L:=\lcm(\ell_0,\ldots,\ell_{R-1}).
\]
The phase vector
\[
  \boldsymbol\beta:=(\beta_0,\ldots,\beta_{R-1})
\]
belongs to \(G_{\boldsymbol A}\), and therefore lies in a unique diagonal orbit \(\cD\subseteq G_{\boldsymbol A}\). Denoting by
\(\boldsymbol\alpha=(\alpha_0,\ldots,\alpha_{R-1})\) the chosen representative of this orbit, we have, for some integer \(t\) determined modulo \(L\),
\[
  \tau^t(\boldsymbol\alpha)=\boldsymbol\beta,
  \qquad\text{that is,}\qquad
  \beta_i=\alpha_i+t
  \quad\text{in }\mathbb Z_{\ell_i}
  \quad(0\le i<R).
\]

Consider now the length-\(n\) window of \(W_{\boldsymbol A,\cD}\) beginning at the round boundary \(Rt\). For \(0\le i<R\) and \(0\le j<m_i\), the coordinate of this window corresponding to the \(j\)-th position of residue class \(i\) is
\[
  W_{\boldsymbol A,\cD}[R(t+j)+i]
  =
  A_i[\alpha_i+t+j]
  =
  A_i[\beta_i+j],
\]
where the phases on the right are read modulo \(\ell_i\). The subsequence of \(W_{\boldsymbol A,\cD}[Rt,n]\) on the coordinates congruent to \(i\) modulo
\(R\) is thus precisely the cyclic window \(A_i[\beta_i,m_i]\). Moreover, the window \(W_{\boldsymbol A,\cD}[Rt,n]\) spans \(m\) complete rounds followed by the first \(s\) symbols of the next round; equivalently, the first \(s\) residue classes contribute \(m+1\) symbols and the remaining \(R-s\) residue
classes contribute \(m\) symbols.

Since the residue classes modulo \(R\) partition the \(n\) coordinates,
\cref{eq:component-cover} yields
\[
  \dist\bigl(y,W_{\boldsymbol A,\cD}[Rt,n]\bigr)
  =
  \sum_{i=0}^{R-1}
  \dist\bigl(y^{(i)},A_i[\beta_i,m_i]\bigr)
  \le R.
\]
Hence every target word \(y\in\Zq^n\) lies within Hamming distance \(R\) of a length-\(n\) window of some member of \(\cP_n\), and therefore \(\cP_n\) is an \((n,R)\)-CSC.
\end{proof}

We next estimate the size of the family \(\cP_n\). The total cyclic length has an exact product form, because the number of diagonal orbits cancels with
their common orbit length. The bound on the number of components is obtained by combining the same orbit decomposition with the short-cycle estimate from
\cref{lem:short-cycles}.

\begin{proposition}\label{prop:interleaving-size}
The total cyclic length of $\cP_n$ is exactly
\begin{equation}\label{eq:product-total-length}
  Q(\cP_n)=RP.
\end{equation}
Moreover, for all sufficiently large \(m\), with
\(h=\lfloor m-\sqrt m\rfloor\),
\begin{equation}\label{eq:product-component-bound}
  D(\cP_n)\le\frac{P}{h}+\left(\frac{q^h}{q-1}\right)^{\!R},
\end{equation}
and consequently
\begin{equation}\label{eq:product-components-asymptotic}
  D(\cP_n)\le\left(1+o_{q,R}(1)\right)\frac{P}{m}.
\end{equation}
\end{proposition}
\begin{proof}
\emph{Total length.} Fix a tuple $\boldsymbol{A}$ with component lengths $\ell_0,\ldots,\ell_{R-1}$ and set $L=\lcm(\ell_0,\ldots,\ell_{R-1})$. This
tuple contributes $(\prod_{i=0}^{R-1}\ell_i)/L$ orbits, each of which is an interleaved sequence of length $RL$, and hence a total length of
\[
  \frac{\prod_{i=0}^{R-1}\ell_i}{L}\cdot RL = R\prod_{i=0}^{R-1}\ell_i,
\]
in which the factor $L$ cancels. Summing over all tuples and factoring the product yields
\[
  Q(\cP_n)=R\prod_{i=0}^{R-1}\!\Bigl(\sum_{A_i\in\cF_i}\len(A_i)\Bigr)
  =R\prod_{i=0}^{R-1}Q_i=RP,
\]
which is \cref{eq:product-total-length}.

\emph{Component count.} The same orbit count furnishes the exact identity
\begin{equation}\label{eq:D-product-exact}
  D(\cP_n)=\sum_{(A_0,\ldots,A_{R-1})}
  \frac{\prod_{i=0}^{R-1}\ell_i}{\lcm(\ell_0,\ldots,\ell_{R-1})}.
\end{equation}
We partition the tuples into two classes. If some \(\ell_i\ge h\), then \(\lcm(\ell_0,\ldots,\ell_{R-1})\ge h\) and the summand is at most $\prod_{i=0}^{R-1}\ell_i/h$; summing over all such tuples contributes at most $P/h$. If instead every $\ell_i<h$, we use \(\lcm(\ell_0,\ldots,\ell_{R-1})\ge1\) and factor:
\begin{equation}\label{eq:all-short-factorization}
  \sum_{\substack{\text{all }\ell_i<h}}\prod_i\ell_i
  =\prod_{i=0}^{R-1}\Bigl(\sum_{\substack{A_i\in\cF_i\\\len(A_i)<h}}
  \len(A_i)\Bigr).
\end{equation}
Applying \cref{lem:short-cycles} to each $\cF_i$ at the common threshold $h$, every factor is less than $q^h/(q-1)$, so this class contributes at most
$(q^h/(q-1))^R$. Adding the two contributions gives
\cref{eq:product-component-bound}.

\emph{Asymptotics.} The lower bound in \cref{eq:Qm-bounds} gives
\begin{equation}\label{eq:P-lower}
   P = \prod_{i=0}^{R-1} Q_i \ge \frac{q^{\sum_{i=0}^{R-1}m_i}}{\prod_{i=0}^{R-1}m_i} = \frac{q^n}{\prod_{i=0}^{R-1}m_i}.
\end{equation}
so that the second term of \cref{eq:product-component-bound}, measured against $P/m$, satisfies
\begin{equation}\label{eq:short-product-ratio}
  \frac{m}{P}\left(\frac{q^h}{q-1}\right)^{\!R} \le \frac{m\prod_{i=0}^{R-1}m_i}{(q-1)^R}\,q^{\,Rh-n}.
\end{equation}
Since $h\le m-\sqrt{m}$ and $n=Rm+s$, we have $Rh-n\le-R\sqrt{m}-s$, whence the right-hand side is a fixed-degree polynomial in $m$ multiplied by
$q^{-R\sqrt{m}}$ and therefore tends to $0$. As \(h/m\to1\), we have
\[
  \frac{P}{h}
  =
  \left(1+o_{q,R}(1)\right)\frac{P}{m} \, ,
\] which yields \cref{eq:product-components-asymptotic}.
\end{proof}

It remains to convert the cyclic family \(\cP_n\) into a single cyclic sequence. To this end, we invoke the standard linearisation and concatenation
procedure for covering sequence codes; see \cite[Section~3]{CheeEtAl2025}.

\begin{lemma}\label{lem:concatenation}
Let \(\cF\) be an \((n,R)\)-CSC. Then there exists a single
\((n,R)\)-covering sequence of length at most
\begin{equation}\label{eq:concatenation-bound}
  Q(\cF)+(n-1)\,D(\cF).
\end{equation}
\end{lemma}
\begin{proof}
For each component \(A\in\cF\), write out one full period of \(A\) and then append its first \(n-1\) symbols, producing a linear block of length
\(\len(A)+n-1\). By construction, every cyclic window \(A[j,n]\) occurs as an ordinary consecutive length-\(n\) subword of this block: should the window wrap around the end of the period, the appended \(n-1\) symbols supply precisely the required continuation.

We now concatenate these linear blocks over all components \(A\in\cF\), in any fixed order, and read the resulting word as a cyclic sequence. Since every
target word in \(\Sigmaq^n\) lies within distance \(R\) of some cyclic \(n\)-window of some component of \(\cF\), and every such cyclic window appears
inside the corresponding linearised block, the concatenated sequence is an \((n,R)\)-covering sequence. Indeed, passing from the linear word to a cyclic
sequence can only introduce further windows across the concatenation boundary, and hence cannot destroy the covering property.

The total length of the concatenated sequence is
\[
  \sum_{A\in\cF}\bigl(\len(A)+n-1\bigr)
  =
  Q(\cF)+(n-1)D(\cF),
\]
which proves the claim.
\end{proof}

We now complete the proof of the main theorem.

\begin{proof}[Proof of the Main Theorem.]
Let \(n=Rm+s\) with \(0\le s<R\), and let \(\cP_n\) be the interleaved
\((n,R)\)-CSC constructed above. By \cref{prop:interleaving-cover}, \(\cP_n\)
possesses the required covering property. Combining the size estimates of
\cref{prop:interleaving-size} with the concatenation bound of
\cref{lem:concatenation}, we obtain
\begin{align}
  \cL_q(n,R)
  &\le Q(\cP_n)+(n-1)D(\cP_n) \notag\\
  &\le RP+(n-1)\left(1+o_{q,R}(1)\right)\frac{P}{m} \notag\\
  &=\left(2R+o_{q,R}(1)\right)P,
  \label{eq:L-vs-P}
\end{align}
where we have used \(n/m=R+\bigO_R(1/m)\).

It remains to estimate \(P\). By \cref{eq:Qm-bounds}, uniformly for
\(m_i\in\{m,m+1\}\),
\[
  Q_i\le\left(q+o_q(1)\right)\frac{q^{m_i}}{m_i}.
\]
As \(R\) is fixed, multiplying these bounds over \(i=0,\ldots,R-1\) yields
\begin{align}
  P=\prod_{i=0}^{R-1}Q_i
  &\le
  \left(q^R+o_{q,R}(1)\right)
  \frac{q^{\sum_{i=0}^{R-1}m_i}}{\prod_{i=0}^{R-1}m_i} \notag\\
  &=
  \left(q^R+o_{q,R}(1)\right)
  \frac{q^n}{m^{R-s}(m+1)^s}.
  \label{eq:P-upper-near-final}
\end{align}
Moreover, since \(m=n/R+\bigO_R(1)\) and \(0\le s<R\),
\begin{equation}\label{eq:mi-product-asymptotic}
  m^{R-s}(m+1)^s
  =
  \left(1+o_R(1)\right)
  \left(\frac{n}{R}\right)^R .
\end{equation}
Substituting this estimate into \cref{eq:P-upper-near-final} gives
\begin{equation}\label{eq:P-final}
  P
  \le
  \left(q^R R^R+o_{q,R}(1)\right)
  \frac{q^n}{n^R}.
\end{equation}
Finally, inserting \cref{eq:P-final} into \cref{eq:L-vs-P} yields
\[
  \cL_q(n,R)
  \le
  \left(2q^R R^{R+1}+o_{q,R}(1)\right)
  \frac{q^n}{n^R},
\]
which is precisely the bound asserted in the Main Theorem. The equivalent volume formulation then follows from the fixed-radius ball-volume asymptotic
\cref{eq:ball-asymptotic}.
\end{proof}

\section*{Acknowledgements}
The work of Hoang Ta was funded by the Ministry of Education and Training of Vietnam under project code CT2025.EA.BKA.08.

\bibliographystyle{unsrt}
\bibliography{references}

\end{document}